\documentclass[11pt]{article}
\usepackage[margin=0.9in]{geometry}
\usepackage[T1]{fontenc}
\usepackage{lmodern}
\usepackage{amsmath,amssymb,amsthm,mathtools,mathrsfs}
\usepackage{microtype}
\usepackage{booktabs,tabularx,array}
\usepackage[numbers,sort&compress]{natbib}
\usepackage{xcolor}
\usepackage{enumitem}
\usepackage{xurl}
\usepackage{hyperref}
\hypersetup{hypertexnames=false,colorlinks=true,linkcolor=blue!45!black,citecolor=blue!45!black,
  urlcolor=blue!45!black,pdftitle={Gradient-Only Online Convex Optimization with a Single Quadratic Leader},
  pdfauthor={Kamiar Asgari}}
\newtheorem{theorem}{Theorem}
\newtheorem{corollary}[theorem]{Corollary}

\theoremstyle{remark}

\setlist[enumerate]{leftmargin=*,itemsep=0.45em}
\allowdisplaybreaks[1]

\title{Gradient-Only Online Convex Optimization\\with a Single Quadratic Leader}
\author{Kamiar Asgari\\\small\href{mailto:kamiaras@usc.edu}{\texttt{kamiaras@usc.edu}}}
\date{September 17, 2026}

\begin{document}
\maketitle
\begin{abstract}
We study online convex optimization on a bounded Euclidean domain when the
learner receives only one subgradient at its prediction and knows neither the
horizon, a gradient bound, nor the losses' strong-convexity parameters.
We present a scale-invariant algorithm that couples projected adaptive gradient
descent to one cone-constrained quadratic leader through a positive wealth
process. The algorithm maintains $O(d)$ state, without a learning-rate grid or
restarts. For a domain of diameter at most $D$, it guarantees regret at most
$\sqrt{2}\,D(\sum_t\|\widehat g_t\|^2)^{1/2}
+2D\max_t\|\widehat g_t\|$, with no logarithmic factor.
A simultaneous quadratic-distance bound yields explicit improvements for
unknown strong convexity, including a refinement governed by jumps in the
observed gradient scale. For varying per-round curvature, the bound depends on
the total shortfall below a reference level selected in hindsight.
The proof uses a clipping-compatible surrogate and a potential involving the
quadratic minimum, log wealth, and two scalar curvature statistics.
The resulting strongly convex bound is not uniform at the classical logarithmic
rate over all curvature-to-gradient ratios; the exact parameter dependence is
stated. The contribution is a single-leader construction and its analysis,
within an established literature on universal and scale-free online learning.
\end{abstract}

\noindent\textbf{Declaration of generative AI use.}
This manuscript, including its mathematical content and exposition, was
generated in its entirety using OpenAI's ChatGPT (GPT-6 Astra) at the request
of Kamiar Asgari.

\section{Introduction}\label{sec:intro}
Online convex optimization asks a learner to make sequential decisions whose
cumulative loss competes with every fixed decision in a convex set. General
convex losses and strongly convex losses admit different regret scales
\citep{Zinkevich2003OGD,Hazan2007Logarithmic}. A
learner that adapts to curvature should retain the general convex guarantee
while exploiting additional curvature whenever it is present. The problem
becomes more delicate when even the magnitude of future gradients is unknown
and feedback is restricted to one subgradient at the point played.

This paper studies that feedback model with a known finite diameter bound.
Our objective is an explicit construction with one evolving quadratic
optimization problem and linear memory. A projected adaptive-gradient
reference sequence supplies the general convex guarantee, following the
adaptive-gradient tradition \citep{Duchi2011AdaGrad,scOrabonaPal2018Scale}.
A quadratic leader
over a cone supplies a comparator-dependent quadratic-distance guarantee.
A positive scalar wealth process connects the two, and affine rescaling
incorporates new gradient maxima without restarting the algorithm.
The construction requires no loss values, curvature observations, gradients
at unplayed points, horizon, or gradient-scale parameter.

The main results are as follows.
\begin{enumerate}
\item \textbf{A logarithm-free convex bound.}
For every comparator, the algorithm achieves regret at most
\[
\sqrt{2}\,D\sqrt{\sum_t\|\widehat g_t\|^2}
+2D\max_t\|\widehat g_t\|.
\]
This statement is deterministic,
simultaneous over comparators, and valid for arbitrary finite observed
subgradients.
\item \textbf{An explicit route from quadratic distance to curvature.}
Theorem~\ref{thm:main} gives a family of bounds indexed by a positive analysis
parameter. Strong convexity cancels the distance term. A sharper version
charges squared increases of the running gradient maximum against the
remaining curvature. The analysis parameter is never supplied to the learner.
\item \textbf{Unknown, varying strong convexity.}
Theorem~\ref{thm:varying} bounds regret by a uniform-curvature expression plus
the sum of per-round curvature shortfalls below a freely chosen reference
level. This applies to nonmonotone curvature sequences with zero-curvature
rounds, while preserving the convex guarantee.
\item \textbf{One quadratic leader and linear state.}
The quadratic is represented by two vectors and three scalars. The full
algorithm uses $O(d)$ memory, one exact cone-constrained quadratic
minimization, and one Euclidean projection per round. The minimization may
cost more than a projection; linear memory is not a claim of linear running
time on an arbitrary domain.
\end{enumerate}

The proof hinges on two cancellations. First, the wealth identity compares
the actual predictions to the adaptive-gradient reference without a
logarithmic penalty. Second, when linear regret is positive, that same
identity caps the wealth that must be paid for in the quadratic potential.
A surrogate coefficient chosen after observing the gradient makes a single
stability argument valid even when the betting vector is clipped.
The relevant Hessian has an isotropic block and a scalar Schur complement,
allowing the potential to depend on two scalar statistics instead of a
dimension-dependent log determinant.

The scope of the result is deliberately specific. Universal regret rates,
single-gradient feedback, scale adaptation, and curvature-shortfall measures
all have substantial precedents, discussed in Section~\ref{sec:related}.
We do not claim a uniform $O(\widehat G_T^2\log(eT)/\mu)$ guarantee for every
ratio $\mu D/\widehat G_T$, nor do we give an exp-concave guarantee.
The point is the particular single-quadratic construction, its explicit
constants and scale-jump refinement, and the guarantees it supports.

Section~\ref{sec:problem} specifies the model; Section~\ref{sec:algorithm}
gives the algorithm. Section~\ref{sec:guarantees} states the bounds and
Section~\ref{sec:proof} proves them. Section~\ref{sec:varying} treats varying
curvature, and Section~\ref{sec:discussion} discusses the limits of the
construction.

\section{Related work}\label{sec:related}
Table~\ref{tab:comparison} summarizes the closest comparisons; the following
subsections discuss their guarantees, assumptions, and feedback models.
\begin{table}[!htbp]
\centering
\small
\setlength{\tabcolsep}{5pt}
\renewcommand{\arraystretch}{1.22}
\caption{Selected close comparisons. Orders in $T$ suppress fixed diameter,
gradient, and curvature parameters; they do not identify the underlying
assumptions. SC denotes strong convexity, EC exp-concavity, and
$k_\mu=\sum_t(1-\mu_t/\mu)_+$. The final row uses observed gradients
and retains the qualifications of Theorems~\ref{thm:main} and
\ref{thm:varying}.}
\label{tab:comparison}
\vspace{0.4em}
\begin{tabularx}{\textwidth}{@{}>{\raggedright\arraybackslash}p{0.18\textwidth}
  >{\raggedright\arraybackslash}p{0.28\textwidth}
  >{\raggedright\arraybackslash}p{0.20\textwidth}
  >{\raggedright\arraybackslash}X@{}}
\toprule
Method & Guarantee or scope & Feedback and scale & Representation or oracle \\
\midrule
MetaGrad \citep{vanErven2021MetaGrad}
& $O(\sqrt{T\log\log T})$ convex; dimension-dependent SC; EC adaptation
& One gradient; unknown scale
& Learning-rate grid; full-matrix, sketch, or coordinate variants \\
\addlinespace
Maler \citep{Wang2020Maler}
& $O(\sqrt T)$ convex; $O(\log T)$ SC; $O(d\log T)$ EC
& One gradient; supplied global gradient bound
& $O(\log T)$ experts \\
\addlinespace
One-projection universal OCO \citep{Yang2024OneProjection}
& Optimal convex, SC, and EC orders
& One gradient; supplied global gradient bound
& Expert family; one projection onto the original domain \\
\addlinespace
Contamination analysis of Maler \citep{scKamijimaIto2024}
& $O(\log T+\sqrt{k_\mu\log T})$ for varying SC
& Same feedback and scale inputs as Maler
& Reanalysis of the existing ensemble \\
\addlinespace
This work
& $O(D\sqrt{\widehat S_T})$ convex; explicit SC and curvature-shortfall bounds
& One subgradient; unknown scale and curvature
& $O(d)$ state; one cone quadratic solve and one projection \\
\bottomrule
\end{tabularx}
\end{table}

\subsection{Universality through multiple learning rates}

Adaptation to unknown curvature is an established objective in online convex
optimization.  MetaGrad aggregates a geometric family of learning rates and
controls a comparator-dependent quadratic variation
$\sum_t\langle g_t,x_t-u\rangle^2$
\citep{vanErven2016MetaGrad,vanErven2021MetaGrad}.  Its full-matrix analysis
gives fast rates for exp-concave losses and favorable stochastic problems;
the corresponding strongly convex guarantee carries a dimension factor.
The general convex guarantee has an additional $\sqrt{\log\log T}$ factor.
The expanded treatment also adapts to unknown gradient scales, maintains
only logarithmically many active learning rates, and offers sketching and
coordinatewise variants.  Thus neither played-point gradient feedback nor
adaptation to gradient magnitude is specific to the present method.

The closest classical comparison for the simultaneous convex and strongly
convex objectives is Maler \citep{Wang2020Maler}.  It combines a convex expert
with families of strongly convex and exp-concave experts, obtaining
$O(\sqrt T)$, $O(\log T)$, and $O(d\log T)$ regret, respectively, when the
other problem parameters are fixed.  Its strongly convex surrogates use
$\eta^2G^2\|x-x_t\|^2$, which is closely related to the squared-distance
term controlled here.  All experts can be updated from the same gradient
at the played point; multiple experts do not imply multiple gradient
queries.  The stated construction uses a gradient bound and a horizon-based
learning-rate grid.  Consequently, simultaneous rate adaptation and removal
of a logarithmic factor from convex regret are already known achievements.

Our structural distinction is to optimize one quadratic over the
homogenized domain, with a continuous comparison parameter $a$ used only
in the analysis.  No discrete learning-rate family is maintained.  This
gives $O(d)$ stored state and a common potential for the convex safeguard
and the squared-distance guarantee.  These properties should be distinguished
from a claim to dominate the complete regret guarantees of universal
ensembles: the present analysis does not supply an exp-concave guarantee,
and its strongly convex dependence is stated explicitly rather than hidden
inside a claim of uniform minimax optimality.

\subsection{Second-order aggregation and stronger feedback}

Second-order expert bounds based on excess losses, rather than raw losses,
provide an important route to curvature adaptation
\citep{Gaillard2014Excess}.  In universal OCO, a squared excess-loss term
can be absorbed by the negative term furnished by strong convexity or
exp-concavity.  This cancellation principle also explains why a generic
$O(\sqrt T)$ aggregation overhead would be inadequate for preserving
logarithmic regret.

The USC framework of \citet{Zhang2022USC} uses a second-order
meta-algorithm to combine algorithms acting on the original losses.
It can inherit specialized guarantees without constructing a new surrogate
for every base algorithm.  This flexibility changes the feedback comparison:
its original-loss experts may require gradients at their own decisions,
and the paper explicitly allows multiple queries per round.  The journal
extension \citep{Zhang2025SecondOrder} additionally treats composite losses
and optimistic aggregation.  We therefore distinguish adaptation of a
collection of original-loss solvers from the present restriction to one
arbitrary subgradient at the played decision.

\subsection{Smoothness, gradient variation, and projection efficiency}

Universal algorithms can exploit information beyond curvature.  UFO
\citep{Wang2020UFO} obtains small-loss guarantees for convex, exp-concave,
and strongly convex smooth losses.  Such results use additional smoothness
and loss assumptions; the squared-gradient quantity in our nonsmooth
guarantee should not be identified with a comparator's cumulative loss
without a separate self-bounding argument.

\citet{Yan2023GradientVariation} develop universal gradient-variation
guarantees with a multilayer ensemble and a one-gradient surrogate
implementation.  \citet{Yan2024OptimalVariation} subsequently obtain the
optimal orders $O(\sqrt{V_T})$, $O(d\log V_T)$, and $O(\log V_T)$ for
the three curvature classes, with additive constants understood, using
$O(\log T)$ base learners and one gradient query.  Here
$V_T=\sum_{t=2}^T\sup_x\|\nabla f_t(x)-\nabla f_{t-1}(x)\|^2$
measures changes in loss functions, rather than gradient energy along the
learner's path.  Their assumptions include bounded gradients and smoothness;
the latter is imposed on the ambient space or a suitable enlargement of the
domain.  Our results make no smoothness assumption and do not claim
gradient-variation adaptation.

The later UniGrad treatment \citep{Zhao2025UniGrad} develops correction-based
and Bregman-divergence-based variants, their one-gradient implementations,
and an anytime ensemble construction.  It further confirms that one-gradient
feedback and an unknown horizon can coexist with universal ensemble methods.
The relevant comparison here is the representation of adaptation by a single
quadratic and the associated assumptions, rather than the availability of
one gradient query in isolation.

Gradient-scale adaptation under generalized smoothness is studied by
\citet{Xie2024GeneralizedSmoothness}.  Their Lipschitz-adaptive
meta-algorithm avoids restarts, but their universal construction uses
multiple base learners, function information beyond the played-point
gradient, and queryable generalized-smoothness information.  These
requirements separate that construction from our feedback model.
Likewise, absence of restarts alone is not a novelty claim.

\citet{Yang2024OneProjection} obtain simultaneous optimal regret orders
with one projection onto the original feasible set per round by operating
experts on a simpler containing domain.  The present method also performs
one projection, but additionally solves a cone-constrained quadratic
problem.  An arbitrary-domain implementation is therefore not comparable
on projection count alone: our $O(d)$ memory bound does not imply an
$O(d)$ running-time bound or a reduction to one ordinary projection.

Across these comparisons, a global gradient bound used in prior algorithm
design must also be distinguished from the observed maximum gradient in
our statements.  The substantive contribution claimed here is the explicit
single-quadratic construction and its simultaneous inequalities under the
stated feedback model.

\subsection{Unknown gradient scales and parameter-free learning}

Scale invariance and adaptation to unknown gradient magnitudes have a substantial
history. \citet{scOrabonaPal2018Scale} develop scale-free FTRL and mirror-descent
methods whose predictions are unchanged by a common positive rescaling of the
loss vectors. Their results already establish that knowledge of a gradient bound
is unnecessary for optimal-order convex regret on bounded domains. Thus the
logarithm-free convex guarantee in this paper is a baseline to preserve while
adding curvature adaptation, rather than a new convex rate.

There is a related, but distinct, literature on simultaneously adapting to
gradient scale and to an unbounded comparator norm.
\citet{scCutkosky2019Hints} remove prior knowledge of both quantities using hints
and artificial constraints. \citet{scMhammediKoolen2020} develop FreeGrad and
reductions that address missing hints and dependence on the ratio between
gradient scales. \citet{scJacobsenCutkosky2022} give a parameter-free
mirror-descent framework, including a scale-free construction that avoids
restarts. \citet{scZhangEtAl2024Discretization} sharpen the dependence on
cumulative squared gradient norms using a refined discretization and also
consider an unknown Lipschitz constant. \citet{scCutkoskyMhammedi2024} obtain
further guarantees when neither the comparator norm nor the Lipschitz constant
is known. These results concern additional comparator adaptation on unbounded
domains. Here the diameter bound is supplied, and the target is simultaneous
adaptation to observed gradient magnitudes and unobserved loss curvature.
In particular, absence of restarts alone is not a distinguishing contribution.

The quantity $J=\sum_t(B_t-B_{t-1})^2$ in our bound measures increases in a
running maximum. It should not be confused with temporal gradient variation.
For example, \citet{scZhaoEtAl2026Variation} study comparator-dependent gradient
variation for smooth losses on an unbounded domain. Our guarantee assumes no
smoothness and can have small $J$ even when gradient directions oscillate.

\subsection{Betting, surrogate losses, and constrained reductions}

The multiplicative wealth update belongs to the coin-betting approach to
parameter-free learning developed by \citet{scOrabonaPal2016Betting}. In that
framework, lower bounds on wealth yield regret guarantees through convex
duality; the standard formulations normalize the coin outcomes using a known
bound. \citet{scCutkoskyOrabona2018} provide reductions from parameter-free
learning to exp-concave optimization, from general Banach spaces to one
dimension, and from constrained to unconstrained optimization.
\citet{scHoevenEtAl2018} connect exponential weights, quadratic surrogates,
gradient methods, and betting-based adaptive algorithms. These works supply
the appropriate context for using a quadratic auxiliary problem and a wealth
potential together.

The reference process also has a connection with the general objective of
combining guarantees. \citet{scCutkosky2019Combining} show how adding the
predictions of suitable unconstrained learners produces an infimal-convolution
regret bound, and develop related optimistic reductions. Our constrained
prediction is instead formed through the cone
$\{(v,a):a\geq0,\ v\in aK\}$ and a denominator that simultaneously enforces
feasibility and clips the betting fraction. The analysis couples this
prediction to an adaptive projected-gradient reference process. The specific
features established here are the clipping-compatible coefficient $b_t$, the
affine rescaling of both auxiliary states when the observed scale grows, and
the joint quadratic--wealth invariant. They should be assessed as a particular
construction within these established frameworks.

\subsection{Heterogeneous and contaminated curvature}

Adaptation to nonuniform curvature predates universal algorithms.
\citet{scBartlettEtAl2007} give adaptive online gradient descent with guarantees
that interpolate according to accumulated curvature. Their update uses the
curvature of losses after they are observed; this differs from the single
subgradient feedback imposed here. In the bandit setting,
\citet{scLuoEtAl2022} allow heterogeneous curvature revealed after each action.
More recently, \citet{scCharikarEtAl2026} study curvature-adaptive FTPL, including
nonconvex losses, using observed loss functions and an approximate offline
optimization oracle. Their cumulative-curvature guarantees address a different
information model.

The closest antecedent for the shortfall interpretation is
\citet{scKamijimaIto2024}. Their Theorem~5.6 already uses
$k_\mu=\sum_t(1-\mu_t/\mu)_+$.
For Maler, their analysis yields
$O((G^2/\mu+GD)\log T+GD\sqrt{k_\mu\log T})$
under their bounded-gradient assumptions. Their universal algorithms use
gradient information; additional function-class information enables sharper
contamination guarantees. Thus neither curvature shortfall nor adaptation to
contaminated strong convexity is new here.
Our result gives an explicit shortfall inequality for the present scale-free
construction, expressed using realized gradients and holding alongside its
convex bound. It does not replace per-round curvature by the average curvature
or by curvature of the final cumulative loss.

These comparisons also delimit the uniform strongly convex claim. The
displayed guarantees retain additive terms involving $D\widehat G_T$ and
$\mu D^2$, and do not establish a uniform
$O(\widehat G_T^2\log(eT)/\mu)$ bound for every ratio
$\mu D/\widehat G_T$. Moreover, linear memory does not imply linear running
time: solving the cone-constrained quadratic can be more costly than the
projection used by the reference process.

\section{Problem formulation}\label{sec:problem}

Let $\mathcal X\subseteq\mathbb R^d$ be a nonempty closed convex set with known finite diameter bound $D>0$. On round $t$, the learner plays $\widehat x_t\in\mathcal X$ and observes one subgradient $\widehat g_t\in\partial\ell_t(\widehat x_t)$ of a convex loss. It receives no curvature parameter, gradient bound, or horizon, and needs neither loss values nor gradients at other points. The gradients may be arbitrary finite vectors. For each fixed comparator $\widehat u\in\mathcal X$, its regret is

\[
R_T(\widehat u)=\sum_{t=1}^T[\ell_t(\widehat x_t)-\ell_t(\widehat u)].
\]

The aim is a convex guarantee without a logarithmic factor, together with improved guarantees that exploit strong convexity, including when its strength varies over time. In the varying-curvature problem, each loss has a possibly different strong-convexity lower bound $\mu_t\ge0$, meaning

\[
\ell_t(u)\ge\ell_t(x)+\langle g,u-x\rangle+\frac{\mu_t}{2}\|u-x\|^2
\quad(x,u\in\mathcal X,\ g\in\partial\ell_t(x)).
\]

The sequence $(\mu_t)$ may vary arbitrarily, need not be monotone, and may include zeros. These lower bounds are never supplied to the learner, including after a round; feedback remains one subgradient at the played point. The regret comparator remains a single fixed $\widehat u$. The uniform strongly convex setting is the special case in which $\mu_t\ge\mu$ for every $t$, for some unknown common $\mu>0$.

These are per-round assumptions on the entire domain, not merely assumptions about curvature at the played point or about the final cumulative loss. The varying-curvature guarantee below measures the total shortfall of $\mu_t$ below a reference level chosen in hindsight, while preserving the convex guarantee.

Fix $x_{\rm ref}\in\mathcal X$ and normalize the domain:

\[
K=(\mathcal X-x_{\rm ref})/D,\qquad
f_t(x)=\ell_t(x_{\rm ref}+Dx),\qquad g_t=D\widehat g_t.
\]

Then $0\in K$, $\operatorname{diam}(K)\le1$, and $\|x\|\le1$ on $K$. Playing $x_t\in K$ means playing $\widehat x_t=x_{\rm ref}+Dx_t$. Regret is unchanged, and normalized per-round strong convexity is $\alpha_t=\mu_tD^2$; in the uniform case, the common lower bound is $\alpha=\mu D^2$. Write

\[
\mathcal R_T(u)=\sum_t[f_t(x_t)-f_t(u)],\qquad
\mathcal L_T(u)=\sum_t\langle g_t,x_t-u\rangle.
\]

Convexity gives $\mathcal R_T(u)\le\mathcal L_T(u)$; the uniform strong-convexity assumption gives

\begin{equation}
\mathcal R_T(u)\le\mathcal L_T(u)-\frac{\alpha}{2}\sum_t\|x_t-u\|^2.
\tag{1}\label{eq:1}
\end{equation}

More generally, the varying-curvature assumption gives

\begin{equation}
\mathcal R_T(u)\le\mathcal L_T(u)-\frac12\sum_t\alpha_t\|x_t-u\|^2.
\tag{1a}\label{eq:1a}
\end{equation}

All norms are Euclidean and logarithms natural. Projection onto $K$ is denoted by $\Pi_K$. The geometric optimization problems below are assumed to be solved exactly. A singleton decision set is trivial.

\section{A single quadratic leader}\label{sec:algorithm}

Use the cone and initial quadratic

\[
\mathcal C=\{(v,a):a\ge0,\ v\in aK\},\qquad
Q_0(v,a)=2(\|v\|^2+a^2).
\]

Initialize $W_0=1$, $z_1=0$, and $B_0=S_0=0$. At round $t$:

\begin{enumerate}
\item Compute the quadratic leader and prediction:

\[
(v_t,a_t)=\arg\min_{\mathcal C}Q_{t-1},\qquad
e_t=v_t-a_tz_t,\qquad
d_t=\max\{1+W_{t-1}a_t,\ 2\|e_t\|\},
\]

\begin{equation}
\beta_t=\frac{e_t}{d_t},\qquad
\omega_t=\frac1{d_t-W_{t-1}a_t},\qquad
\boxed{x_t=z_t+W_{t-1}\beta_t.}
\tag{2}\label{eq:2}
\end{equation}

These formulas also apply when $a_t=0$ or $e_t=0$.

\item Play $x_t$, observe $g_t$, and set

\[
B_t=\max\{B_{t-1},\|g_t\|\},\qquad S_t=S_{t-1}+\|g_t\|^2.
\]

If $B_t=0$, set $p_t=b_t=0$, leave $Q,W,z$ unchanged, and continue.

\item Otherwise set

\[
h_t=g_t/B_t,\qquad p_t=\|h_t\|,\qquad
q_t=-\langle h_t,\beta_t\rangle,\qquad \lambda_t=B_{t-1}/B_t,
\]

\begin{equation}
\boxed{b_t=\max\{p_t^2,\ \omega_t^2p_t^2+
4\omega_t(1-\omega_t)q_t\}.}
\tag{3}\label{eq:3}
\end{equation}

Form the single quadratic surrogate

\[
s_t(v,a)=\langle h_t,v-ax_t\rangle+b_t\|v-ax_t\|^2
\]

and update

\begin{equation}
\boxed{\begin{aligned}
Q_t&=Q_0+\lambda_t(Q_{t-1}-Q_0+s_t),\\
W_t&=1+\lambda_t\bigl(W_{t-1}(1+q_t)-1\bigr),\\
z_{t+1}&=\Pi_K\left(z_t-\frac{g_t}{\sqrt{2S_t}}\right).
\end{aligned}}
\tag{4}\label{eq:4}
\end{equation}

\end{enumerate}

The coefficient $b_t$ is fixed when forming $s_t$; it is not recomputed at the surrogate's argument. Its meaning is simple: it is the smallest coefficient at least $p_t^2$ for which

\begin{equation}
s_t(v_t,a_t)\ge
\langle h_t,\beta_t\rangle+p_t^2\|\beta_t\|^2.
\tag{5}\label{eq:5}
\end{equation}

The proof establishes this characterization and $p_t^2\le b_t\le p_t/(2-p_t)\le p_t$. Thus additional quadratic weight is used only when the prediction clipping requires it.

There is no scale cap, learning-rate grid, or restart. The first nonzero gradient initializes the scale through $\lambda_t=0$. Multiplying all gradients by a positive constant leaves the predictions unchanged.

For implementation, maintain two vectors and three scalars:

\[
Q_t(v,a)=Q_0(v,a)+\langle A_t,v\rangle-a c_t+
N_t\|v\|^2-2a\langle X_t,v\rangle+a^2E_t.
\]

Starting from zero, their updates on a nonzero prefix are

\[
\begin{aligned}
A_t&=\lambda_t(A_{t-1}+h_t),&
c_t&=\lambda_t(c_{t-1}+\langle h_t,x_t\rangle),\\
N_t&=\lambda_t(N_{t-1}+b_t),&
X_t&=\lambda_t(X_{t-1}+b_tx_t),\\
E_t&=\lambda_t(E_{t-1}+b_t\|x_t\|^2).
\end{aligned}
\]

Including the reference point $z_t$, the memory cost is $O(d)$. Each round uses one cone-constrained quadratic minimization and one projection. The former need not be as inexpensive as a projection for an arbitrary domain.

\section{Regret guarantees}\label{sec:guarantees}

Fix $T\ge1$ and let $G=B_T$. If $G=0$, regret is nonpositive. Otherwise define

\[
n=\frac1G\sum_tB_{t-1}b_t,\qquad
m=\frac1G\sum_tB_{t-1}p_t,\qquad
\Gamma=\max_tB_{t-1}b_t,\qquad
J=\sum_t(B_t-B_{t-1})^2.
\]

Then $0\le n\le m\le T$, $\Gamma\le G$, and $J\le G^2$. Put

\[
\gamma=3-4\log2,\qquad L=\log(1+n/2),\qquad
\vartheta=\frac{n}{n+2},
\]

\begin{equation}
\Phi=\min_{\kappa\in[1/4,1]}
\left\{2\kappa L+\frac{2(1-\kappa)^2}{\gamma}\vartheta\right\},
\qquad
\boxed{H=\Phi+\log\min\left\{1+m,\ 2+\frac{\sqrt{2S_T}}G\right\}.}
\tag{6}\label{eq:6}
\end{equation}

For $n>0$, the minimizing value is
$\kappa=\max\{1/4,1-\gamma L/(2\vartheta)\}$.
For $n=0$, also $m=\Gamma=H=0$. These are analysis quantities, not algorithm parameters. In particular,

\[
H\le\frac12\log(1+n/2)+\frac9{8\gamma}\vartheta+
\log\min\left\{1+m,\ 2+\frac{\sqrt{2S_T}}G\right\}
\]

\begin{equation}
\le \frac12\log(1+T/2)+\log(2+\sqrt{2T})+\frac9{8\gamma}
=\log T+O(1).
\tag{7}\label{eq:7}
\end{equation}

\begin{theorem}[Simultaneous regret bounds]\label{thm:main} Simultaneously for every $u\in K$,

\begin{equation}
\boxed{\mathcal R_T(u)\le\sqrt{2S_T}+2G
\le(\sqrt2+2)\sqrt{S_T}.}
\tag{8}\label{eq:8}
\end{equation}

Let $C(u)=2(1+\|u\|^2)\le4$. For every $a>0$,

\begin{equation}
\boxed{\mathcal L_T(u)\le
G+C(u)Ga+\frac{GH}{a}
+a\sum_tB_{t-1}b_t\|x_t-u\|^2.}
\tag{9}\label{eq:9}
\end{equation}

If every loss is $\alpha$-strongly convex, then

\begin{equation}
\boxed{\mathcal R_T(u)\le
G+C(u)Ga+\frac{GH}{a}
\quad(a>0,\ 2a\Gamma\le\alpha),}
\tag{10}\label{eq:10}
\end{equation}

and the sharper gradient-jump bound is

\begin{equation}
\boxed{\mathcal R_T(u)\le
C(u)Ga+\frac{GH}{a}
+\min\left\{G,\frac{J}{2(\alpha-2a\Gamma)}\right\}
\quad(a>0,\ 2a\Gamma<\alpha).}
\tag{11}\label{eq:11}
\end{equation}

Also $\mathcal R_T(u)\le S_T/(2\alpha)$. All applicable bounds hold simultaneously. For $H>0$, \eqref{eq:10} is minimized at
$a=\min\{\sqrt{H/C(u)},\alpha/(2\Gamma)\}$, interpreting division by $\Gamma=0$ as infinity. If $H=0$, take $a\downarrow0$ in \eqref{eq:11} to obtain $\mathcal R_T(u)\le J/(2\alpha)$; do not substitute zero into $H/a$.

\end{theorem}

\subsection{Original coordinates and interpretation}
In original coordinates, write

\[
\widehat S_T=\sum_t\|\widehat g_t\|^2,\quad
\widehat G_T=\max_t\|\widehat g_t\|,\quad
\widehat\Gamma=\Gamma/D,\quad \widehat J=J/D^2.
\]

The dimensionless quantities $n,m,H$ are unchanged. The uniform bounds become

\begin{equation}
R_T(\widehat u)\le\sqrt2\,D\sqrt{\widehat S_T}+2D\widehat G_T,
\tag{12}\label{eq:12}
\end{equation}

\begin{equation}
R_T(\widehat u)\le D\widehat G_T(1+4a+H/a)
\quad(a>0,\ 2a\widehat\Gamma\le\mu D),
\tag{13}\label{eq:13}
\end{equation}

\begin{equation}
R_T(\widehat u)\le4D\widehat G_Ta+\frac{D\widehat G_TH}{a}
+\min\left\{D\widehat G_T,
\frac{\widehat J}{2(\mu-2a\widehat\Gamma/D)}\right\}
\quad(a>0,\ 2a\widehat\Gamma<\mu D).
\tag{14}\label{eq:14}
\end{equation}

A simpler consequence, using $\widehat\Gamma\le\widehat G_T$, is

\begin{equation}
\boxed{R_T(\widehat u)\le
\begin{cases}
D\widehat G_T+2\mu D^2+\dfrac{2\widehat G_T^2}{\mu}H,
&\mu D\le\widehat G_T\sqrt H,\\[1mm]
D\widehat G_T(1+4\sqrt H),
&\mu D\ge\widehat G_T\sqrt H.
\end{cases}}
\tag{15}\label{eq:15}
\end{equation}

The convex guarantee has no logarithmic factor. The strongly convex guarantee uses $H\le\log T+O(1)$, but still does not establish a uniform $O(\widehat G_T^2\log(eT)/\mu)$ bound for arbitrarily large $\mu D/\widehat G_T$.

\section{Proof of the main theorem}\label{sec:proof}

\subsection{Well-definedness and the clipping surrogate}
The cone is closed and convex because $K$ is closed, convex, and bounded. Each $Q_t-Q_0$ is linear plus a nonnegative sum of squared norms. Thus $Q_t$ is coercive and $4$-strongly convex, and its cone minimizer exists and is unique.

Suppress the round index and write $W$ for the wealth before the update. The denominator in \eqref{eq:2} satisfies $d\ge1+Wa$. If $a>0$, the prediction is the convex combination
$x=(1-Wa/d)z+(Wa/d)(v/a)$; if $a=0$, cone feasibility gives $v=0$ and $x=z$. Hence $x\in K$. More importantly, putting $r=v-ax$ gives

\begin{equation}
r=(d-Wa)\beta,\qquad
\beta=\omega r,\qquad 0<\omega\le1,\qquad \|\beta\|\le\tfrac12.
\tag{16}\label{eq:16}
\end{equation}

When $\omega<1$, the second term defines $d$, so $\|\beta\|=1/2$. Thus $\beta$ is exactly $r$ clipped to radius $1/2$, but the algorithm never needs to compute $r$.

Since $|q|\le p/2$,

\begin{equation}
p^2\le b\le
\max\{p^2,\ \omega^2p^2+2p\omega(1-\omega)\}
\le\frac{p}{2-p}\le p.
\tag{17}\label{eq:17}
\end{equation}

The middle bound follows by maximizing the quadratic in $\omega\in[0,1]$. If $\omega=1$, then $b=p^2$ and \eqref{eq:5} is an equality. Otherwise $\|\beta\|^2=1/4$ and

\[
4\omega^2\left[
\langle h,r\rangle+b\|r\|^2-
\langle h,\beta\rangle-p^2\|\beta\|^2
\right]
=b-\omega^2p^2-4\omega(1-\omega)q\ge0.
\]

This proves \eqref{eq:5}, including the minimality of its coefficient. The intermediate wealth is $W^\circ=W(1+q)\ge W/2>0$; affine rescaling preserves positivity. Projection preserves $z\in K$. These observations justify the algorithm inductively, starting from $W_0=1$.

\subsection{The scale correction and the convex guarantee}
If $G=0$, every $g_t$ is zero and convexity gives $\mathcal R_T(u)\le0$; hence assume $G>0$ for the remainder of the proof. Define $\bar g_t=(B_{t-1}/B_t)g_t$ when $B_t>0$, and zero otherwise. Then

\begin{equation}
\|g_t-\bar g_t\|=B_t-B_{t-1},\qquad
\sum_t\|g_t-\bar g_t\|=G.
\tag{18}\label{eq:18}
\end{equation}

Multiplying the wealth update by $B_t$ and summing yields the exact cancellation

\begin{equation}
G(W_T-1)=\sum_t\langle\bar g_t,z_t-x_t\rangle.
\tag{19}\label{eq:19}
\end{equation}

Indeed, each increment is
$B_tW_t-B_{t-1}W_{t-1}
=B_t-B_{t-1}+\langle\bar g_t,z_t-x_t\rangle$.
Consequently $W_T\le1+m$.

The reference update, with $\eta_t=1/\sqrt{2S_t}$ on nonzero prefixes, satisfies

\[
\langle g_t,z_t-u\rangle
\le
\frac{\|z_t-u\|^2-\|z_{t+1}-u\|^2}{2\eta_t}
+\frac{\eta_t\|g_t\|^2}{2}.
\]

The first terms sum to at most $1/(2\eta_T)=\sqrt{S_T/2}$ because inverse step sizes are nondecreasing and squared distances are at most one. The second terms also sum to at most $\sqrt{S_T/2}$, using
$(S_t-S_{t-1})/\sqrt{S_t}\le2(\sqrt{S_t}-\sqrt{S_{t-1}})$.
Zero prefixes contribute nothing. Combining this reference bound with \eqref{eq:18}--\eqref{eq:19} gives

\begin{equation}
\boxed{\mathcal L_T(u)\le\sqrt{2S_T}+2G-GW_T.}
\tag{20}\label{eq:20}
\end{equation}

Dropping the final negative term proves \eqref{eq:8}. Retaining it will save half a logarithm later.

\subsection{Quadratic stability through two scalar statistics}
We now control the quadratic leader. For any current state write

\[
\mathsf H=\nabla^2Q=
\begin{pmatrix}F I&-2X\\-2X^\top&4+2E\end{pmatrix},
\qquad
F=4+2N,\quad
\rho=4+2E-\frac{4\|X\|^2}{F},\quad
\bar x=\frac{2X}{F}.
\]

The historical weights are nonnegative, so $\bar x\in K$ and $4\le\rho\le F$. The first statement expresses $\bar x$ as a convex combination including $0$; the second follows from the weighted variance identity and $\|x\|\le1$.

A prime denotes the intermediate update $Q'=Q+s$, before rescaling. Direct calculation gives

\begin{equation}
F'=F+2b,\qquad
\rho'-\rho=2b\frac F{F'}\|x-\bar x\|^2\le2b\le2,\qquad
x-\bar x'=\frac F{F'}(x-\bar x).
\tag{21}\label{eq:21}
\end{equation}

At the current cone minimizer let $k=h+2br$, and write $M=\min_{\mathcal C}Q$ and $M'=\min_{\mathcal C}Q'$. Constrained first-order optimality followed by completion of the square gives

\[
M'\ge M+s(v,a)
-\tfrac12(k,-\langle k,x\rangle)^\top(\mathsf H')^{-1}
(k,-\langle k,x\rangle)
\ge M+s(v,a)-\mathcal A\|k\|^2,
\]

where block inversion gives

\begin{equation}
\mathcal A=\frac1{2F'}+\frac{\|x-\bar x'\|^2}{2\rho'},
\qquad \tau=b\mathcal A,\qquad Z=2\tau.
\tag{22}\label{eq:22}
\end{equation}

To bound $\tau$, put $M_x=(I,-x)$. Since
$\mathsf H'\succeq4I+2bM_x^\top M_x$ and $\|M_x\|^2\le2$,

\begin{equation}
\tau=b\,\lambda_{\max}\!\left(\tfrac12M_x(\mathsf H')^{-1}M_x^\top\right)
\le\frac{b}{4+4b}\le\frac18.
\tag{23}\label{eq:23}
\end{equation}

Two elementary telescoping estimates will pay for stability:

\begin{equation}
2Z\le\Delta\log F+\Delta\log\rho,\qquad
Z^2\le\Delta(-1/F)+\Delta(-1/\rho).
\tag{24}\label{eq:24}
\end{equation}

Here $\Delta$ refers only to the intermediate update. For the first, use
$Z=b/F'+b\|x-\bar x'\|^2/\rho'$ and \eqref{eq:21} to bound $2Z$ by
$\Delta F/F'+\Delta\rho/\rho'$.
For the second, call the two summands $Z_1,Z_2$.
Since $\Delta F,\Delta\rho\le2$, \eqref{eq:21} gives
$Z_1^2\le\frac12\Delta(-1/F)$ and
$Z_2^2\le\frac12\Delta(-1/\rho)$; now use
$(Z_1+Z_2)^2\le2Z_1^2+2Z_2^2$.

\subsection{The potential and affine rescaling}
The essential simplification is to retain the cross term in
$\|h+2br\|^2$. If $b=0$, also $h=0$ and the intermediate state is unchanged. Otherwise set
$y=\langle h,\beta\rangle=-q$ and $w=p\|\beta\|\le1/2$.
The possible decrease in $M+\log W$ is at most

\begin{equation}
\begin{aligned}
\mathscr D
&=\mathcal A\|h+2br\|^2-\langle h,r\rangle-b\|r\|^2-\log(1-y)\\
&=\tau\frac{p^2}{b}
-(1-4\tau)\bigl(\langle h,r\rangle+b\|r\|^2\bigr)-\log(1-y)\\
&\le\tau-(1-4\tau)(y+w^2)-\log(1-y).
\end{aligned}
\tag{25}\label{eq:25}
\end{equation}

The last step uses $p^2\le b$, $1-4\tau\ge0$, and precisely the surrogate inequality \eqref{eq:5}. No separate clipped-round analysis remains.

For $|y|\le1/2$,
$-\log(1-y)\le y+(4\log2-2)y^2$:
the ratio $[-\log(1-y)-y]/y^2$ is at most its value at $y=1/2$.
Together with $|y|\le w$, this turns \eqref{eq:25} into

\[
\mathscr D\le\tau(1+2w)^2-\gamma w^2.
\]

For any $\kappa\in[1/4,1]$,

\[
(1+2w)^2-4\kappa-8(1-\kappa)w
=(2w-1)(2w+4\kappa-1)\le0.
\]

Completing one square therefore gives the single bound

\begin{equation}
\mathscr D
\le4\kappa\tau+\frac{16(1-\kappa)^2}{\gamma}\tau^2
=2\kappa Z+\frac{4(1-\kappa)^2}{\gamma}Z^2.
\tag{26}\label{eq:26}
\end{equation}

Define, for $s\ge4$,

\[
\psi_\kappa(s)=\kappa\log(s/4)
+\frac{4(1-\kappa)^2}{\gamma}(1/4-1/s).
\]

Equations \eqref{eq:24}--\eqref{eq:26} imply that every intermediate update preserves

\begin{equation}
\boxed{\min_{\mathcal C}Q+\log W+
\psi_\kappa(F)+\psi_\kappa(\rho)\ge0.}
\tag{27}\label{eq:27}
\end{equation}

Initially there is equality. To include affine rescaling, denote the rescaled state by a plus. Since $\min Q_0=0$ and logarithms are concave,

\[
\min Q^+\ge\lambda\min Q',\qquad
\log W^+\ge\lambda\log W'.
\]

Moreover,

\[
F^+=4+\lambda(F'-4),\qquad
\rho^+\ge4+\lambda(\rho'-4).
\]

In the latter inequality the difference is
$16\lambda(1-\lambda)\|X'\|^2/(F'F^+)\ge0$.
Each $\psi_\kappa$ is increasing, concave, and zero at $4$, so its rescaled values are at least $\lambda$ times their intermediate values. Adding proves \eqref{eq:27} after rescaling as well, including $\lambda=0$.

\subsection{Comparator substitution and curvature cancellation}
At the final time, $N_T=n$, $F_T=4+2n$, and $\rho_T\le F_T$. Optimizing \eqref{eq:27} over $\kappa$ consequently gives the unconditional state bound

\begin{equation}
\min_{\mathcal C}Q_T\ge-\Phi-\log W_T.
\tag{28}\label{eq:28}
\end{equation}

Fix a comparator $u$. If $\mathcal L_T(u)\le0$, \eqref{eq:9} holds immediately, and the remaining bounds of the preceding theorem follow since their right sides are nonnegative and $\mathcal R_T(u)\le\mathcal L_T(u)\le0$. Otherwise \eqref{eq:20} implies

\[
W_T<2+\frac{\sqrt{2S_T}}G.
\]

Together with $W_T\le1+m$, this makes \eqref{eq:28} at least $-H$. This is the second cancellation: large wealth already pays for regret, so its logarithm need not be charged at the crude scale $T$. The capped bound $\min Q_T\ge-H$ is used only in this positive-linear-regret case.

Unrolling the quadratic recursion gives

\begin{equation}
Q_T(v,a)=Q_0(v,a)+\frac1G\sum_t
\left[\langle\bar g_t,v-ax_t\rangle+
B_{t-1}b_t\|v-ax_t\|^2\right].
\tag{29}\label{eq:29}
\end{equation}

The feasible ray $(v,a)=(au,a)$ turns this into a scalar quadratic. Since
$Q_0(au,a)=C(u)a^2$ and $Q_T(au,a)\ge-H$, division by $a>0$ yields

\begin{equation}
\sum_t\langle\bar g_t,x_t-u\rangle
\le C(u)Ga+\frac{GH}{a}
+a\sum_tB_{t-1}b_t\|x_t-u\|^2.
\tag{30}\label{eq:30}
\end{equation}

Restoring $g_t$ costs at most $G$ by \eqref{eq:18}, proving \eqref{eq:9}.
Subtracting the strong-convexity term in \eqref{eq:1} cancels the last sum whenever $2a\Gamma\le\alpha$, proving \eqref{eq:10}.

For the sharper bound, retain
$c=\alpha-2a\Gamma>0$ units of curvature. The remaining terms satisfy

\[
\langle g_t-\bar g_t,x_t-u\rangle-\frac c2\|x_t-u\|^2
\le\frac{(B_t-B_{t-1})^2}{2c}.
\]

Their sum is also at most $G$ by dropping the nonpositive quadratic term. Taking the smaller estimate proves \eqref{eq:11}. Independently, \eqref{eq:1} and completion of the square on each round give
$f_t(x_t)-f_t(u)\le\|g_t\|^2/(2\alpha)$, hence the bound $S_T/(2\alpha)$.

If $n=0$, every $B_{t-1}b_t$ vanishes. Since $b_t=0$ exactly when $p_t=0$, this forces $m=\Gamma=0$, so $H=0$ and the stated $a\downarrow0$ limit applies. Finally, substituting
$G=D\widehat G_T$, $S_T=D^2\widehat S_T$,
$\Gamma=D\widehat\Gamma$, $J=D^2\widehat J$,
and $\alpha=\mu D^2$ proves \eqref{eq:12}--\eqref{eq:14}.
Minimizing $4a+H/a$ under $2aG\le\alpha$ proves \eqref{eq:15}. This completes the proof of Theorem~\ref{thm:main}.

\section{Unknown per-round strong convexity}\label{sec:varying}

\begin{theorem}[Unknown per-round strong convexity]\label{thm:varying} Suppose each $\ell_t$ satisfies the problem definition with a strong-convexity lower bound $\mu_t\ge0$, and run the algorithm \eqref{eq:2}--\eqref{eq:4} without modification. If $G=0$, then $R_T(\widehat u)\le0$ for every comparator. Otherwise, use the quantities $G,H,C(u)$ defined above, put

\[
\alpha_t=\mu_tD^2,\qquad w_t=B_{t-1}b_t,\qquad (z)_+=\max\{z,0\}.
\]

Simultaneously for every $u\in K$ and every $a>0$,

\begin{equation}
\begin{aligned}
\mathcal R_T(u)
&\le G+C(u)Ga+\frac{GH}{a}
+\sum_t\left(aw_t-\frac{\alpha_t}{2}\right)\|x_t-u\|^2\\
&\le G+C(u)Ga+\frac{GH}{a}
+\sum_t\left(aw_t-\frac{\alpha_t}{2}\right)_+.
\end{aligned}
\tag{31}\label{eq:31}
\end{equation}

In original coordinates, for every reference level $\bar\mu>0$,

\begin{equation}
\boxed{R_T(\widehat u)\le
D\widehat G_T+2\bar\mu D^2
+\frac{2\widehat G_T^2H}{\bar\mu}
+\frac{D^2}{2}\sum_{t=1}^T(\bar\mu-\mu_t)_+.}
\tag{32}\label{eq:32}
\end{equation}

Neither $a$ nor $\bar\mu$ is an algorithm parameter. Both can be optimized in hindsight, and the convex bounds \eqref{eq:8} and \eqref{eq:12} hold simultaneously. In particular, one may take the smaller of \eqref{eq:12} and the infimum of the right side of \eqref{eq:32} over $\bar\mu>0$. Here $H\le\log T+O(1)$ as in \eqref{eq:7}; the statements also cover $H=0$ with every displayed denominator strictly positive.

\end{theorem}

The extra term in \eqref{eq:32} charges only rounds with $\mu_t<\bar\mu$. If at most $k$ rounds fall below $\bar\mu$, their total contribution is at most $k\bar\mu D^2/2$, since $\mu_t\ge0$. Thus weak or zero-curvature rounds incur a controlled penalty without replacing all per-round curvatures by their minimum. This is a guarantee in terms of curvature shortfalls; it does not replace a common curvature lower bound by the arithmetic average of the $\mu_t$.

\begin{proof} The zero-gradient case follows from convexity, so assume $G>0$. Subtract the varying-curvature term in \eqref{eq:1a} from the linear-regret bound \eqref{eq:9}. This gives the first inequality in \eqref{eq:31}. For each real $c$ and $0\le r^2\le1$, one has $cr^2\le(c)_+$. Since $\operatorname{diam}(K)\le1$, apply this with $c=aw_t-\alpha_t/2$ and $r=\|x_t-u\|$ to obtain the second inequality.

For \eqref{eq:32}, fix $\bar\mu>0$ and choose

\[
a=\frac{\bar\mu D^2}{2G}
=\frac{\bar\mu D}{2\widehat G_T}>0.
\]

By \eqref{eq:17}, $0\le b_t\le p_t\le1$, so $0\le w_t\le G$. Consequently,

\[
\left(aw_t-\frac{\alpha_t}{2}\right)_+
\le\left(aG-\frac{\mu_tD^2}{2}\right)_+
=\frac{D^2}{2}(\bar\mu-\mu_t)_+.
\]

Also, $C(u)\le4$ and $G=D\widehat G_T$ give

\[
C(u)Ga\le2\bar\mu D^2,\qquad
\frac{GH}{a}=\frac{2\widehat G_T^2H}{\bar\mu}.
\]

Substituting these estimates into \eqref{eq:31} proves \eqref{eq:32}. The simultaneous convex guarantee follows from the original theorem, which requires only convexity. No part of this argument supplies the curvature sequence to the algorithm. \end{proof}

\subsection{A finite number of weak-curvature rounds}
The following consequence makes the relation to contaminated OCO explicit.
The shortfall interpretation and its use for contamination are already present
in \citet{scKamijimaIto2024}; the constants below follow from \eqref{eq:32}.

\begin{corollary}[At most $k$ weak-curvature rounds]\label{cor:contamination}
Suppose $\mu_t\ge\mu>0$ on all but at most $k$ rounds, and $\mu_t\ge0$
on the remaining rounds. Write $\widehat G=\widehat G_T$.
If $\widehat G>0$ and $H>0$, define
\[
 r_* = \min\left\{\mu,\frac{2\widehat G}{D}
                  \sqrt{\frac{H}{k+4}}\right\}.
\]
Then, simultaneously with the convex bound,
\begin{equation}
 R_T(\widehat u)\le D\widehat G+
 \frac{D^2(k+4)}{2}r_*+\frac{2\widehat G^2H}{r_*}.
 \tag{33}\label{eq:33}
\end{equation}
In particular,
\begin{equation}
 R_T(\widehat u)\le D\widehat G+
 \frac{2\widehat G^2H}{\mu}
 +2D\widehat G\sqrt{(k+4)H}.
 \tag{34}\label{eq:34}
\end{equation}
If $\widehat G>0$ and $H=0$, the bound $R_T(\widehat u)\le D\widehat G$
holds instead. If $\widehat G=0$, regret is nonpositive.
\end{corollary}
\begin{proof}
For $0<r\le\mu$, at most $k$ terms of $\sum_t(r-\mu_t)_+$ are nonzero,
and each is at most $r$. Equation~\eqref{eq:32} therefore gives
\[
R_T(\widehat u)\le D\widehat G+Ar+B/r,
\qquad A=\frac{D^2(k+4)}2,\quad B=2\widehat G^2H.
\]
When $H>0$, the minimizer on $(0,\mu]$ is
$\min\{\mu,\sqrt{B/A}\}=r_*$. If $\mu\ge\sqrt{B/A}$, the optimized
sum is $2\sqrt{AB}$. Otherwise $A\mu\le\sqrt{AB}$, so it is at most
$B/\mu+2\sqrt{AB}$. This proves \eqref{eq:33}--\eqref{eq:34}.
When $H=0$, take $r\downarrow0$ without substituting zero into a denominator.
The zero-gradient case follows from convexity.
\end{proof}

\section{Discussion and limitations}\label{sec:discussion}
The algorithm combines a general convex guarantee and curvature-sensitive
bounds in a single run. Its distinctive representation is the cone quadratic:
an isotropic vector block, a vector--scalar cross term, and one scalar
quadratic coefficient suffice. The storage cost does not grow with the
number of candidate learning rates, the horizon, or the number of gradient
scale increases. Positive rescaling of every gradient leaves the predictions
unchanged, because it leaves normalized gradients and scale ratios unchanged
and cancels out of the reference update.

Several limits are essential to interpreting the guarantees. First, the
diameter bound is known, and projections and cone minimizations are exact.
The analysis does not quantify accumulated optimization error or establish
a uniform running-time advantage over expert-based methods. Second, the
uniform-curvature bounds retain additive geometric terms. Even together
with $S_T/(2\alpha)$, they do not establish the classical logarithmic rate
uniformly for arbitrarily large curvature relative to the largest observed
gradient. Third, the varying-curvature result uses per-round strong
convexity on the entire domain. Strong convexity only of the final cumulative
loss, or only near the played points, is insufficient for the stated proof.

The curvature-shortfall quantity records weak rounds relative to a chosen
reference level. It does not justify replacing all per-round curvatures by
their arithmetic average. Its value is a posteriori adaptation under the
same single-subgradient feedback model, with the convex bound available
simultaneously. Prior work already studies this form of contamination;
the relationship is spelled out in Section~\ref{sec:related}.

The present work is theoretical and contains no empirical performance claim.
Natural next questions are whether the remaining geometric terms can be
removed without introducing a learning-rate grid, which domains admit
particularly efficient cone minimization, and how to extend the potential
to approximate geometric oracles or richer curvature classes.

{\small
\bibliographystyle{plainnat}
\bibliography{universal,scale}

\begin{thebibliography}{30}
\providecommand{\natexlab}[1]{#1}
\providecommand{\url}[1]{\texttt{#1}}
\expandafter\ifx\csname urlstyle\endcsname\relax
  \providecommand{\doi}[1]{doi: #1}\else
  \providecommand{\doi}{doi: \begingroup \urlstyle{rm}\Url}\fi

\bibitem[Bartlett et~al.(2007)Bartlett, Hazan, and Rakhlin]{scBartlettEtAl2007}
Peter~L. Bartlett, Elad Hazan, and Alexander Rakhlin.
\newblock Adaptive online gradient descent.
\newblock In \emph{Advances in Neural Information Processing Systems},
  volume~20, 2007.
\newblock URL
  \url{https://papers.neurips.cc/paper_files/paper/2007/hash/261b909dfbee5a1a09d5eb50ed7a17e0-Abstract.html}.

\bibitem[Charikar et~al.(2026)Charikar, Pabbaraju, and
  Tewari]{scCharikarEtAl2026}
Moses Charikar, Chirag Pabbaraju, and Ambuj Tewari.
\newblock From non-convex to strongly convex: Curvature-adaptive {FTPL} for
  online optimization, 2026.
\newblock URL \url{https://arxiv.org/abs/2606.02948}.

\bibitem[Cutkosky(2019{\natexlab{a}})]{scCutkosky2019Combining}
Ashok Cutkosky.
\newblock Combining online learning guarantees.
\newblock In \emph{Proceedings of the Thirty-Second Conference on Learning
  Theory}, volume~99 of \emph{Proceedings of Machine Learning Research}, pages
  895--913. PMLR, 2019{\natexlab{a}}.
\newblock URL \url{https://proceedings.mlr.press/v99/cutkosky19b.html}.

\bibitem[Cutkosky(2019{\natexlab{b}})]{scCutkosky2019Hints}
Ashok Cutkosky.
\newblock Artificial constraints and hints for unbounded online learning.
\newblock In \emph{Proceedings of the Thirty-Second Conference on Learning
  Theory}, volume~99 of \emph{Proceedings of Machine Learning Research}, pages
  874--894. PMLR, 2019{\natexlab{b}}.
\newblock URL \url{https://proceedings.mlr.press/v99/cutkosky19a.html}.

\bibitem[Cutkosky and Mhammedi(2024)]{scCutkoskyMhammedi2024}
Ashok Cutkosky and Zakaria Mhammedi.
\newblock Fully unconstrained online learning.
\newblock In \emph{Advances in Neural Information Processing Systems},
  volume~37, 2024.
\newblock URL
  \url{https://proceedings.neurips.cc/paper_files/paper/2024/hash/13cd22c32c1330decd69c13cf8cadc0a-Abstract-Conference.html}.

\bibitem[Cutkosky and Orabona(2018)]{scCutkoskyOrabona2018}
Ashok Cutkosky and Francesco Orabona.
\newblock Black-box reductions for parameter-free online learning in {Banach}
  spaces.
\newblock In \emph{Proceedings of the 31st Conference on Learning Theory},
  volume~75 of \emph{Proceedings of Machine Learning Research}, pages
  1493--1529. PMLR, 2018.
\newblock URL \url{https://proceedings.mlr.press/v75/cutkosky18a.html}.

\bibitem[Duchi et~al.(2011)Duchi, Hazan, and Singer]{Duchi2011AdaGrad}
John Duchi, Elad Hazan, and Yoram Singer.
\newblock Adaptive subgradient methods for online learning and stochastic
  optimization.
\newblock \emph{Journal of Machine Learning Research}, 12\penalty0
  (61):\penalty0 2121--2159, 2011.
\newblock URL \url{https://jmlr.org/papers/v12/duchi11a.html}.

\bibitem[Gaillard et~al.(2014)Gaillard, Stoltz, and van
  Erven]{Gaillard2014Excess}
Pierre Gaillard, Gilles Stoltz, and Tim van Erven.
\newblock A second-order bound with excess losses.
\newblock In \emph{Proceedings of the 27th Conference on Learning Theory},
  volume~35 of \emph{Proceedings of Machine Learning Research}, pages 176--196.
  PMLR, 2014.
\newblock URL \url{https://proceedings.mlr.press/v35/gaillard14.html}.

\bibitem[Hazan et~al.(2007)Hazan, Agarwal, and Kale]{Hazan2007Logarithmic}
Elad Hazan, Amit Agarwal, and Satyen Kale.
\newblock Logarithmic regret algorithms for online convex optimization.
\newblock \emph{Machine Learning}, 69\penalty0 (2--3):\penalty0 169--192, 2007.
\newblock \doi{10.1007/s10994-007-5016-8}.
\newblock URL
  \url{https://link.springer.com/article/10.1007/s10994-007-5016-8}.

\bibitem[Jacobsen and Cutkosky(2022)]{scJacobsenCutkosky2022}
Andrew Jacobsen and Ashok Cutkosky.
\newblock Parameter-free mirror descent.
\newblock In \emph{Proceedings of Thirty Fifth Conference on Learning Theory},
  volume 178 of \emph{Proceedings of Machine Learning Research}, pages
  4160--4211. PMLR, 2022.
\newblock URL \url{https://proceedings.mlr.press/v178/jacobsen22a.html}.

\bibitem[Kamijima and Ito(2024)]{scKamijimaIto2024}
Tomoya Kamijima and Shinji Ito.
\newblock Contaminated online convex optimization.
\newblock \emph{Transactions on Machine Learning Research}, 2024.
\newblock ISSN 2835-8856.
\newblock URL \url{https://openreview.net/forum?id=QdGtwjDgub}.

\bibitem[Luo et~al.(2022)Luo, Zhang, and Zhao]{scLuoEtAl2022}
Haipeng Luo, Mengxiao Zhang, and Peng Zhao.
\newblock Adaptive bandit convex optimization with heterogeneous curvature.
\newblock In \emph{Proceedings of Thirty Fifth Conference on Learning Theory},
  volume 178 of \emph{Proceedings of Machine Learning Research}, pages
  1576--1612. PMLR, 2022.
\newblock URL \url{https://proceedings.mlr.press/v178/luo22a.html}.

\bibitem[Mhammedi and Koolen(2020)]{scMhammediKoolen2020}
Zakaria Mhammedi and Wouter~M. Koolen.
\newblock {Lipschitz} and comparator-norm adaptivity in online learning.
\newblock In \emph{Proceedings of Thirty Third Conference on Learning Theory},
  volume 125 of \emph{Proceedings of Machine Learning Research}, pages
  2858--2887. PMLR, 2020.
\newblock URL \url{https://proceedings.mlr.press/v125/mhammedi20a.html}.

\bibitem[Orabona and P{\'a}l(2016)]{scOrabonaPal2016Betting}
Francesco Orabona and D{\'a}vid P{\'a}l.
\newblock Coin betting and parameter-free online learning.
\newblock In \emph{Advances in Neural Information Processing Systems},
  volume~29, 2016.
\newblock URL
  \url{https://proceedings.neurips.cc/paper/2016/hash/320722549d1751cf3f247855f937b982-Abstract.html}.

\bibitem[Orabona and P{\'a}l(2018)]{scOrabonaPal2018Scale}
Francesco Orabona and D{\'a}vid P{\'a}l.
\newblock Scale-free online learning.
\newblock \emph{Theoretical Computer Science}, 716:\penalty0 50--69, 2018.
\newblock \doi{10.1016/j.tcs.2017.11.021}.
\newblock URL \url{https://arxiv.org/abs/1601.01974}.

\bibitem[van~der Hoeven et~al.(2018)van~der Hoeven, van Erven, and
  Kot{\l}owski]{scHoevenEtAl2018}
Dirk van~der Hoeven, Tim van Erven, and Wojciech Kot{\l}owski.
\newblock The many faces of exponential weights in online learning.
\newblock In \emph{Proceedings of the 31st Conference on Learning Theory},
  volume~75 of \emph{Proceedings of Machine Learning Research}, pages
  2067--2092. PMLR, 2018.
\newblock URL \url{https://proceedings.mlr.press/v75/hoeven18a.html}.

\bibitem[van Erven and Koolen(2016)]{vanErven2016MetaGrad}
Tim van Erven and Wouter~M. Koolen.
\newblock {MetaGrad}: Multiple learning rates in online learning.
\newblock In \emph{Advances in Neural Information Processing Systems},
  volume~29, 2016.
\newblock URL
  \url{https://papers.nips.cc/paper_files/paper/2016/hash/14cfdb59b5bda1fc245aadae15b1984a-Abstract.html}.

\bibitem[van Erven et~al.(2021)van Erven, Koolen, and van~der
  Hoeven]{vanErven2021MetaGrad}
Tim van Erven, Wouter~M. Koolen, and Dirk van~der Hoeven.
\newblock {MetaGrad}: Adaptation using multiple learning rates in online
  learning.
\newblock \emph{Journal of Machine Learning Research}, 22\penalty0
  (161):\penalty0 1--61, 2021.
\newblock URL \url{https://jmlr.org/papers/v22/20-1444.html}.

\bibitem[Wang et~al.(2020{\natexlab{a}})Wang, Lu, Hu, and Zhang]{Wang2020UFO}
Guanghui Wang, Shiyin Lu, Yao Hu, and Lijun Zhang.
\newblock Adapting to smoothness: A more universal algorithm for online convex
  optimization.
\newblock \emph{Proceedings of the AAAI Conference on Artificial Intelligence},
  34\penalty0 (04):\penalty0 6162--6169, 2020{\natexlab{a}}.
\newblock \doi{10.1609/aaai.v34i04.6081}.
\newblock URL \url{https://ojs.aaai.org/index.php/AAAI/article/view/6081}.

\bibitem[Wang et~al.(2020{\natexlab{b}})Wang, Lu, and Zhang]{Wang2020Maler}
Guanghui Wang, Shiyin Lu, and Lijun Zhang.
\newblock Adaptivity and optimality: A universal algorithm for online convex
  optimization.
\newblock In \emph{Proceedings of the 35th Uncertainty in Artificial
  Intelligence Conference}, volume 115 of \emph{Proceedings of Machine Learning
  Research}, pages 659--668. PMLR, 2020{\natexlab{b}}.
\newblock URL \url{https://proceedings.mlr.press/v115/wang20e.html}.
\newblock Conference held in 2019.

\bibitem[Xie et~al.(2024)Xie, Zhao, and Zhou]{Xie2024GeneralizedSmoothness}
Yan-Feng Xie, Peng Zhao, and Zhi-Hua Zhou.
\newblock Gradient-variation online learning under generalized smoothness.
\newblock In \emph{Advances in Neural Information Processing Systems},
  volume~37, 2024.
\newblock URL
  \url{https://papers.nips.cc/paper_files/paper/2024/hash/42c40aff7814e9796266e12053b1c610-Abstract-Conference.html}.

\bibitem[Yan et~al.(2023)Yan, Zhao, and Zhou]{Yan2023GradientVariation}
Yu-Hu Yan, Peng Zhao, and Zhi-Hua Zhou.
\newblock Universal online learning with gradient variations: A multi-layer
  online ensemble approach.
\newblock In \emph{Advances in Neural Information Processing Systems},
  volume~36, 2023.
\newblock URL
  \url{https://papers.nips.cc/paper_files/paper/2023/hash/76818d8d85e05e45ce3a16a8468619d1-Abstract-Conference.html}.

\bibitem[Yan et~al.(2024)Yan, Zhao, and Zhou]{Yan2024OptimalVariation}
Yu-Hu Yan, Peng Zhao, and Zhi-Hua Zhou.
\newblock A simple and optimal approach for universal online learning with
  gradient variations.
\newblock In \emph{Advances in Neural Information Processing Systems},
  volume~37, 2024.
\newblock URL
  \url{https://papers.nips.cc/paper_files/paper/2024/hash/150f35763dc51bfc269690d36a5a7c88-Abstract-Conference.html}.

\bibitem[Yang et~al.(2024)Yang, Wang, Zhao, and Zhang]{Yang2024OneProjection}
Wenhao Yang, Yibo Wang, Peng Zhao, and Lijun Zhang.
\newblock Universal online convex optimization with {1} projection per round.
\newblock In \emph{Advances in Neural Information Processing Systems},
  volume~37, 2024.
\newblock URL
  \url{https://proceedings.neurips.cc/paper_files/paper/2024/hash/37e90dcf2909b5068858b34b5239f187-Abstract-Conference.html}.

\bibitem[Zhang et~al.(2022)Zhang, Wang, Yi, and Yang]{Zhang2022USC}
Lijun Zhang, Guanghui Wang, Jinfeng Yi, and Tianbao Yang.
\newblock A simple yet universal strategy for online convex optimization.
\newblock In \emph{Proceedings of the 39th International Conference on Machine
  Learning}, volume 162 of \emph{Proceedings of Machine Learning Research},
  pages 26605--26623. PMLR, 2022.
\newblock URL \url{https://proceedings.mlr.press/v162/zhang22af.html}.

\bibitem[Zhang et~al.(2025)Zhang, Wang, Wang, Yi, and
  Yang]{Zhang2025SecondOrder}
Lijun Zhang, Yibo Wang, Guanghui Wang, Jinfeng Yi, and Tianbao Yang.
\newblock Universal online convex optimization meets second-order bounds.
\newblock \emph{Journal of Machine Learning Research}, 26\penalty0
  (152):\penalty0 1--53, 2025.
\newblock URL \url{https://jmlr.org/papers/v26/24-1131.html}.

\bibitem[Zhang et~al.(2024)Zhang, Yang, Cutkosky, and
  Paschalidis]{scZhangEtAl2024Discretization}
Zhiyu Zhang, Heng Yang, Ashok Cutkosky, and Ioannis~C. Paschalidis.
\newblock Improving adaptive online learning using refined discretization.
\newblock In \emph{Proceedings of the 35th International Conference on
  Algorithmic Learning Theory}, volume 237 of \emph{Proceedings of Machine
  Learning Research}, pages 1208--1233. PMLR, 2024.
\newblock URL \url{https://proceedings.mlr.press/v237/zhang24a.html}.

\bibitem[Zhao et~al.(2025)Zhao, Yan, Yu, and Zhou]{Zhao2025UniGrad}
Peng Zhao, Yu-Hu Yan, Hang Yu, and Zhi-Hua Zhou.
\newblock Adaptivity and universality: Problem-dependent universal regret for
  online convex optimization, 2025.
\newblock URL \url{https://arxiv.org/abs/2511.19937}.
\newblock Preprint, version 1.

\bibitem[Zhao et~al.(2026)Zhao, Jacobsen, Cesa-Bianchi, and
  Zhao]{scZhaoEtAl2026Variation}
Yuheng Zhao, Andrew Jacobsen, Nicol{\`o} Cesa-Bianchi, and Peng Zhao.
\newblock Gradient-variation regret bounds for unconstrained online learning.
\newblock In \emph{Proceedings of Thirty Ninth Conference on Learning Theory},
  volume 336 of \emph{Proceedings of Machine Learning Research}, pages
  7062--7104. PMLR, 2026.
\newblock URL \url{https://proceedings.mlr.press/v336/zhao26a.html}.

\bibitem[Zinkevich(2003)]{Zinkevich2003OGD}
Martin Zinkevich.
\newblock Online convex programming and generalized infinitesimal gradient
  ascent.
\newblock In \emph{Proceedings of the 20th International Conference on Machine
  Learning}, pages 928--936, 2003.
\newblock URL \url{https://www.cs.cmu.edu/~maz/publications/ICML03.pdf}.

\end{thebibliography}
}
\end{document}